\documentclass[11 pt]{amsart}
\usepackage{amsfonts}
\usepackage{enumerate}
\usepackage{amssymb,latexsym}
\usepackage{amsthm,amsmath}
\usepackage[utf8]{inputenc}\usepackage[english]{babel}
\usepackage[utf8]{inputenc}
\usepackage[mathscr]{euscript}
\usepackage{hyperref}
\usepackage{tikz}
\usepackage{cancel}
\usepackage{graphicx}
\usepackage{multirow}
\usepackage{multicol}
\usepackage{wasysym}
\usepackage[all]{xy}
\usepackage{enumerate}
\usepackage{tabularx}
\usepackage{tabulary} 
\usepackage{tabularht}
\usepackage [dvipsnames] { xcolor } 
\usepackage{colortbl} 
 
\newtheorem{theorem}{Theorem}
\newtheorem{proposition}{Proposition}

\newtheorem{corollary}{Corollary}
\newtheorem{definition}{Definition}
\newtheorem{example}[theorem]{Example}
\newtheorem{remark}[theorem]{Remark}
\usepackage{babel}

\newcommand{\be}{\begin{eqnarray*}}
\newcommand{\ee}{\end{eqnarray*}}

\makeatletter
\@namedef{subjclassname@2020}{%
  \textup{2020} Mathematics Subject Classification}
\makeatother
 
\title[ Bi-atomic systems in Hilbert spaces]{ Bi-atomic systems in Hilbert spaces }
 
\makeatletter
\def\@settitle{\begin{center}{\LARGE\@title\par}\end{center}\vspace{-1ex}}
\def\@setauthors{%
  \begin{center}
  \authors
  \end{center}\vspace{1ex}}
\makeatother
 
\title[Bi-atomic systems in Hilbert spaces]{Bi-atomic systems in Hilbert spaces}
 
\author[Ferrer Villar Osmin  et al.]{%
 Osmin Ferrer Villar\textsuperscript{1,*}\quad
 Jos\'e Naranjo Mart\'inez\textsuperscript{2}\quad
 Breiner Rivero Reyes\textsuperscript{3}\\[6pt]
 \small \textsuperscript{1}Institute Modeling and Decision Sciences (MODS Institute),\\
 \small Monteria, Colombia\\
 \small Orcid iD: https://orcid.org/0000-0003-1015-6271.\\[4pt]
 \small \textsuperscript{2}Center for Basic Sciences, School of Engineering and Architecture,\\
 \small Pontifical Bolivarian University, Monteria, Colombia\\
 \small Orcid iD: https://orcid.org/0000-0003-1685-9314\\[4pt]
 \small \textsuperscript{3}Department of Mathematics, Faculty of Education and Sciences,\\
 \small University of Sucre, Sincelejo, Colombia.\\
 \small Orcid iD: https://orcid.org/0009-0001-8225-8157\\[5pt]
 \small \texttt{osmin.ferrer@modsinstitute.org};\quad
 \texttt{jose.naranjom@upb.edu.co}\\
 \small \texttt{breinerjavierriveroreyes@gmail.com}\\[4pt]
 \small \textsuperscript{*}osmin.ferrer@modsinstitute.org}
 
\begin{document}

\begin{abstract}
In this paper, we introduce and study bi-atomic systems
in Hilbert spaces, a natural generalization of classical
atomic systems; this new concept allows us to represent bounded linear operators using two Bessel sequences simultaneously. Furthermore, we establish equivalent characterizations of these systems,
study their relationship with Bessel bi-frames, and prove
the stability of the class under natural operations:
duality with respect to the adjoint, orthogonal projections,
unitary operators, and concatenation.
\end{abstract}


\maketitle

\section{Introduction}
The representation of elements of a Hilbert space by sequences has been a central topic in functional analysis since the mid-20th century. Orthonormal bases, Riesz bases, and, more recently, frames have provided different ways to decompose any vector in the space in terms of a given family \cite{Casazza, Christensen2003, Duffin, w-marcos en Krein}. Each of these concepts imposes different conditions on the family: uniqueness, stability, redundancy, or exact reconstruction, depending on the requirements of the problem, as in the work of Ferrer, Sierra, and Sanabria \cite{Softframes}, who study soft frames in soft spaces.

A concept that has gained attention over the past two decades is that of the \textit{atomic system for an operator}. Introduced by Feichtinger and Werther \cite{Feichtinger2001}, this structure allows the action of a bounded linear operator \ on any element of the space to be represented by a generating sequence and a sequence of Bessel coefficients, without requiring the family to be a frame. This approach has proven useful in contexts where the structure of the operator is more relevant than that of the generating family.

An extension of atomic systems has been developed by Ferrer, Arroyo, and Naranjo in the context of Krein spaces \cite{FerrerArroyo2023}. This work shows how the notion of an atomic system can be formulated in spaces with indefinite metrics, which extends the scope of application of these structures beyond classical Hilbert spaces. In this paper, however, we focus on Hilbert spaces and on a generalization involving two sequences instead of one: bi-atomic systems.

The manuscript is organized as follows: the first section presents the definitions and classical results that will be used throughout the paper. In the second section, we introduce the definition of a bi-atomic system for a bounded operator, study the symmetry of the concept, and present a concrete example in $\mathbb{C}^{2}$. In the third section, an equivalent characterization of bi-atomic systems is established in terms of the boundedness of the coefficients and the convergence of series. In the final section the most important for the development of this research we study the stability properties of bi-atomic systems under duality, orthogonal projections, unitary operators, and concatenation, and establish their relationship with Bessel biframes. Finally, the conclusions and future directions of this work are presented.

\section{Preliminares}
\label{Sec:2}
In this section, we have compiled the definitions and
classic results that will be used throughout this paper.

\begin{definition}\cite{Kreyszig1991}
Let $(X,\|\cdot \|_X)$ and $(Y,\|\cdot\|_Y)$ be normed spaces over $\mathbb{C}$, and let ${T:X\rightarrow Y}$ be a linear operator. $T$ is said to be \textbf{\textit{bounded}} if there exists a constant $c>0$ such that for every $x\in X$, $\|T(x)\|_{Y}\leq c\|x\|_{X}$ holds.
\end{definition}

\begin{proposition}\cite{Conway1985, Kreyszig1991}\label{prop:adjunto}
Let $(H_1,\langle\cdot,\cdot\rangle_1)$ and $(H_2,\langle\cdot,\cdot\rangle_2)$ be Hilbert spaces, and let $A\in B(H_1,H_2)$. Then there exists a unique operator $A^*\in B(H_2,H_1)$ such that $\text{for all }x\in H_1,\ y\in H_2$
\[
\langle Ax,y\rangle_2=\langle x,A^*y\rangle_1.
\]
Furthermore, $\|A^*\|=\|A\|$.
\end{proposition}

\begin{definition}\cite{Conway1985, Kreyszig1991} The {\bf space of square-summable sequences} is defined by
\begin{eqnarray*}
\ell^2(\mathbb{N}):=\left\lbrace\{\alpha_{n}\}_{n\in\mathbb{N}}\subset\mathbb{C}\big/ \sum_{n\in\mathbb{N}}\vert \alpha_{n}\vert ^{2}<\infty\right\rbrace .
\end{eqnarray*}
which has the structure of a Hilbert space with the inner product
\begin{equation*}
    \langle \cdot,\cdot\rangle:\ell^{2}(\mathbb{N})\times \ell^{2}(\mathbb{N})\rightarrow \mathbb{C},\,\, defined \,\, by
\end{equation*}
\begin{equation*}
    \left\langle \{ \alpha_{n}\}_{n\in\mathbb{N}},\{\beta_{n}\}_{n\in\mathbb{N}}\right\rangle_{\ell^{2}(\mathbb{N})}=\sum_{n\in \mathbb{N}}  \alpha_{n}\overline{\beta_{n}} \,\,\,\,\ for \,\, \{ \alpha_{n}\}_{n\in\mathbb{N}},\{\beta_{n}\}_{n\in\mathbb{N}}\in \ell^{2}(\mathbb{N}).
\end{equation*}
\end{definition}

\begin{remark}
\begin{align*}
 \| \{\alpha_n\}_{n \in \mathbb{N}} \|_{\ell^2(\mathbb{N})}
= \sqrt{\langle  \{\alpha_n\}_{n \in \mathbb{N}}, \{\alpha_n\}_{n \in \mathbb{N}} \rangle_{\ell^{2}(\mathbb{N})}}=\sqrt{\sum_{n \in \mathbb{N}} \alpha_{n}\overline{\alpha_{n}}}=\sqrt{\sum_{n \in \mathbb{N}} |\alpha_{n}|^{2}}.   
\end{align*}

\end{remark}

\begin{definition}\cite{Conway1985}\label{norma l2xl2}
The product space
\[
\ell^2(\mathbb{N})\times\ell^2(\mathbb{N}),
\]
identified with the Hilbert direct sum
\[
\ell^2(\mathbb{N})\oplus\ell^2(\mathbb{N}),
\]
is a Hilbert space when endowed with the inner product
\[
\begin{aligned}
&\left\langle
(\{c_n\}_{n \in \mathbb{N}},\{d_n\}_{n \in \mathbb{N}}),
(\{e_n\}_{n \in \mathbb{N}},\{f_n\}_{n \in \mathbb{N}})
\right\rangle\\
&=
\left\langle \{c_n\}_{n \in \mathbb{N}},\{e_n\}_{n \in \mathbb{N}}\right\rangle_{\ell^2(\mathbb{N})}
+\
\left\langle \{d_n\}_{n \in \mathbb{N}},\{f_n\}_{n \in \mathbb{N}}\right\rangle_{\ell^2(\mathbb{N})}=
\sum_{n\in\mathbb{N}}c_n\overline {e_n}
+
\sum_{n\in\mathbb{N}}d_n\overline{f_n}.
\end{aligned}
\]
The norm induced by this inner product is given by
\[
\left\|(\{c_n\}_{n \in \mathbb{N}},\{d_n\}_{n \in \mathbb{N}})\right\|_{\ell^2(\mathbb{N})\times\ell^2(\mathbb{N})}
=
\left(
\left\|\{c_n\}_{n \in \mathbb{N}}\right\|_{\ell^2(\mathbb{N})}^{2}
+
\left\|\{d_n\}_{n \in \mathbb{N}}\right\|_{\ell^2(\mathbb{N})}^{2}
\right)^{1/2}.
\]

\end{definition}

\begin{theorem}
\cite{Conway1985, Kreyszig1991}
\label{Teorema de Banach-Steinhaus}
    Let $X$ be a Banach space, $Y$ a normed space,
    and \\ $\{T_\alpha\}_{\alpha \in A} \subset \mathbf{B}(X, Y)$
    a family of bounded linear operators.
    If for every $x \in X$ we have
    $\sup\limits_{\alpha}\|{T_\alpha(x)}\| < \infty$,
    then $\sup\limits_{\alpha}\|{T_\alpha}\| < \infty.$

    In particular, if $\{T_k\}_{k \in \mathbb{N}} \subset \mathbf{B}(X,Y)$
    converges pointwise to an operator\\ $T: X \to Y$,
    then $T$ is linear and bounded.
\end{theorem}

\begin{theorem}
\cite{Conway1985, Kreyszig1991}
\label{Teorema-representacion-Riesz}
    Let $(H, \langle{\cdot},{\cdot}\rangle)$ be a Hilbert space
    and $\varphi: H \to \mathbb{C}$ a bounded linear functional.
    Then there exists a unique vector $u \in H$ such that
    \[
        \varphi(x) = \langle{x},{u}\rangle
        \quad \forall\, x \in H.
    \]
    Furthermore, $\|{\varphi}\| = \|{u}\|$.
\end{theorem}

\begin{theorem}\cite{Douglas} \label{thm: Douglas}
  Let  $( H_{1}, \langle\cdot,\cdot \rangle_{1}), ( H_{2}, \langle\cdot, \cdot \rangle_{2})$ 
  and $( H, \langle\cdot,\cdot \rangle) $ are Hilbert spaces and  $T_1\in{B}({H}_1,{H}), T_2\in{B}({H}_2, {H})$ 
  bounded operators. The following statements are equivalent:
  \begin{itemize}
      \item [(i)] ${R}(T_1)\subset{R}(T_2);$
      \item[(ii)] $T_1T_1^*\leq \lambda^2T_2T_2^*$ for some $\lambda \geq 0$ and 
      \item[(iii)] There exist a bounded operator   $X\in {L}({H}_1,H_{2})$ such that $T_1=T_2X.$
  \end{itemize}
  \end{theorem}

  \begin{definition}
\cite{WFE}\label{fK}
 Let $(H, \langle \cdot,\cdot\rangle)$ be a Hilbert space, a sequence $\mathbb{X}=\left\{x_{n}\right\}_{n\in \mathbb{ N}}\subset H$ is called a \emph{Bessel sequence} for the Hilbert space, if there exists a constant $0< \beta< \infty$ such that
\begin{equation}
 \sum_{n\in \mathbb{N}}\left|\langle x,x_{n}\rangle\right|^{2} \,\leq\,
\beta \| x\|^{2}\quad \text{for all} \  \,x\in{K}.  \label{m1}
\end{equation}
\end{definition}

\begin{proposition}\label{resultado sucesiones de Bessel}
\cite{Christensen2003} Let $\left(\mathcal{H},\langle \cdot,\cdot\rangle_{\mathcal{H}}\right)$ be a Hilbert space and $\left\{x_{n}\right\}_{n\in \mathbb{N}}$ a Bessel sequence; the operator
\begin{eqnarray*}
    T:&\ell_{2}(\mathbb{N})&\longrightarrow \mathcal{H}\\
    &\left\{\alpha_{n}\right\}_{n\in \mathbb{N}}&\longrightarrow T(\left\{\alpha_n\right\}_{n\in \mathbb{N}})= \sum_{n\in \mathbb{N}}  \alpha_{n} x_{n}
\end{eqnarray*}
is bounded and well-defined.\\

Furthermore, $T$ has an adjoint given by

\begin{eqnarray*}
T^{\ast}:&\mathcal{H}&\longrightarrow \ell_{2}(\mathbb{N})\\
&x& \longrightarrow T^{\ast}x=\left\{ \langle x,x_{n}\rangle \right\}_{n\in \mathbb{N}}.
\end{eqnarray*}
which is also bounded.
\end{proposition}

\begin{definition}\cite{Feichtinger2001, FerrerArroyo2023}
 Let $(H, \langle \cdot,\cdot\rangle)$ be a Hilbert space $T\in \mathcal{B}({H}).$ We say that $\{x_n\}_{n\in \mathbb{N}}\subset {H}$ is an atomic system for $T$ in $H$ if:
\begin{itemize}

\item [(i)]$\{x_n\}_{n\in \mathbb{N}}$ is  a Bessel sequence for $H$.

\item [(ii)] There exist a Bessel sequence  $\{y_n\}_{n\in \mathbb{N}}$ in the Hilbert space  such that

    $$\displaystyle{T(x)=\sum_{n \in \mathbb{N}}  \langle x, y_n\rangle  x_n} \ \ \text{for all}\ x\in H.$$
    \end{itemize}
\end{definition}

\begin{definition}\label{Def.Marco}\cite{Christensen2003}
 Let $(\mathcal{H}, \langle \cdot, \cdot \rangle )$ be a Hilbert space and
 $\{x_n\}_{n \in \mathbb{N}}$  a sequence  of elements of $\mathcal{H}$. It is called a \emph{frame} for $\mathcal{H}$ if there exist constants $\lambda ,\beta>0$ such that

\begin {equation}
\lambda \|x\|^{2}\leq\sum_{n \in \mathbb{N}}|\langle x,x_n\rangle|^{2}\leq \beta\|x\|^{2},
\qquad \forall\,x\in\mathcal{H}.
\end{equation}
\end{definition}

\begin{definition}\label{def bimarco}\cite{Parizi2023}
Let $(H, \langle \cdot,\cdot\rangle)$ be a Hilbert space.
$\mathbb{X} = \{x_{n}\}_{n \in \mathbb{N}}$, $\mathbb{Y} = \{y_{n}\}_{n \in \mathbb{N}}$ are sequences of points in $\mathcal{H}$.
If there exist constants $0<\lambda\le \beta<\infty$ such that
\[
\lambda \|x\|^{2}
\le
\sum_{n \in \mathbb{N}} \langle x,x_n\rangle \,\langle y_n,x\rangle
\le
\beta \|x\|^{2},
\quad \forall x \in \mathcal{H},
\]
then $(\mathbb{X},\mathbb{Y})$ is said to be a bimarco for $\mathcal{H}$. When the sequences $\mathbb{X}$ and $\mathbb{Y}$ are Bessel sequences, we refer to a Bessel bimarco.
\end{definition}

\begin{definition}\label{Operador-bimarco}\cite{Parizi2022, Parizi2023}
Let $(\mathcal{H}, \langle \cdot, \cdot \rangle )$ be a Hilbert space, $\mathbb{X}=\{x_n\}_{n \in \mathbb{N}}$, $\mathbb{Y}=\{y_n\}_{n \in \mathbb{N}}$ sequences in $\mathcal{H}$, and $(\mathbb{X},\ \mathbb{Y}$ a bi-frame for $(\mathcal{H}, \langle\cdot,\cdot\rangle)$. The \textbf{b-frame operator} $S_{\mathbb{X}\mathbb{Y}}$ is defined as follows:
\begin{align*}
S_{{\mathbb{X}\mathbb{Y}}}:\ & \mathcal{H}\to \mathcal{H}\\
&x \to S_{\mathbb{X}\mathbb{Y}}(x):= \sum_{n \in \mathbb{N}} \langle x,x_{n}\rangle y_{n}.
\end{align*}
\end{definition}

\begin{proposition}\label{propiedades del operador bimarco}\cite{Parizi2022}
  Let $(\mathcal{H}, \langle\cdot \rangle )$ a Hilbert space, $\mathbb{X}=\{x_n\}_{n \in \mathbb{N}}$, $\mathbb{Y}=\{y_n\}_{n \in \mathbb{N}}$ sequences in $\mathcal{H}$, and $(\mathbb{X},\mathbb{Y})$ a bimarco for $(\mathcal{H}, \langle\cdot,\cdot\rangle),$ then the bimarco operator $S_{\mathbb{X}\mathbb{Y}}$ is well-defined, is bounded and positive, and invertible.
\end{proposition}

\begin{remark}\label{obs:adjunto-bimarco}
Let \((\mathcal{H},\langle\cdot,\cdot\rangle)\) be a Hilbert space and \(\mathbb{X}=\{x_n\}_{n\in\mathbb{N}}\), \(\mathbb{Y}=\{y_n\}_{n\in\mathbb{N}}\) be sequences in \(\mathcal{H}\). Then the two-frame operators \(S_{\mathbb{X}\mathbb{Y}}\) and \(S_{\mathbb{Y}\mathbb{X}}\) satisfy $S_{\mathbb{X}\mathbb{Y}} = S_{\mathbb{Y}\mathbb{X}}^{*}.$
\end{remark}

Below, we present the results of this research.

\section{Bi-atomic systems in Hilbert spaces}
\label{sec:3}

\begin{definition}\label{sistema biatomico}
  Let $(H, \langle \cdot, \cdot \rangle)$ be a Hilbert space, $T\in \mathcal{B}({H})$, and let the sequences $\mathbb{X}=
  \{x_n\}_{n \in \mathbb{N}}$, $\mathbb{Y} = \{y_n\}_{n \in \mathbb{N}} \subset H$, then the pair $(\mathbb{X}, \mathbb{Y})$ is a bi-atomic system for $T$ in $H$ if:
 
\begin{enumerate}
      \item $\mathbb{X},$ $\mathbb{Y}$ are Bessel sequences.
      \item There exist $\mathbb{F}=\{f_n\}_{n \in \mathbb{N}}, \mathbb{G}= \{g_n\}_{n \in \mathbb{N}} \subset {H}$ which are Bessel sequences such that $$T(x)=\sum \limits_{n \in \mathbb{N}} \left( \langle x, f_n\rangle x_n + \langle x, g_n\rangle y_n  \right)$$
      We will refer to the sequences $\mathbb{F}, \mathbb{G}$ as auxiliary sequences for $\mathbb{X}, \mathbb{Y}$ respectively.
  \end{enumerate}
\end{definition}

\begin{remark}\label{simetria de biatomicos}
   Let $(H, \langle \cdot, \cdot \rangle)$ be a Hilbert space, $T \in \mathcal{B}({H})$, and let the sequences $\mathbb{X} = \{x_n\}_{n \in \mathbb{N}}$, $\mathbb{Y} = \{y_n\}_{n \in \mathbb{N}} \subset {H}$, then if $(\mathbb{X}, \mathbb{Y})$ is a bi-atomic system for $T$, then $(\mathbb{Y}, \mathbb{X})$ is a bi-atomic system for $T$. The proof follows immediately from the Definition \eqref{sistema biatomico}. 
\end{remark}

\begin{example}\label{ejemplo bi-atomico}
    Let $H=\mathbb{C}^2$ with the usual inner product $\langle (a,b), (c,d) \rangle = a\overline{c}+b\overline{d}$ and $T:H \to H$ defined by $T((a,b))=(2a,6b),$ which is bounded, since for all $x=(a,b) \in H$, we have that:
    $$\Vert T(x) \Vert^2 =\Vert T((a,b)) \Vert^2 = \Vert (2a,6b) \Vert^2 = 4|a|^{2}+36|b|^{2} \leq 36(|a|^{2}+ |b|^{2}) = 36\Vert x \Vert^2.$$
    \vspace{2mm}
   On the other hand, let us consider the following sequences in $\mathbb{C}^2$, $\mathbb{X}=\{x_n\}_{n \in \mathbb{N}}$,  $\mathbb{Y}=\{y_n\}_{n \in \mathbb{N}}$, given by  $x_{1}= (i,0), x_{2}=(0,3i),$ $x_{n}=(0,0),$ $y_{1}=(5i,0),$\\ $y_{2}=(0,i),$ and furthermore $x_{n}= (0,0),$ $y_{n}=(0,0)$ for $n\geq 3.$ Let’s show that $(\mathbb{X},\mathbb{Y})$ is a bi-atomic system for $T.$     Indeed,  for every $x=(a,b) \in \mathbb{C}^2,$ we have that:
        \begin{itemize}
            \item $\langle x,x_{1}\rangle = \langle (a,b),(i,0)\rangle = -ai$
            \item $\langle x_{2},x\rangle = \langle (a,b),(0,3i)\rangle = -3bi$
            \item $\langle x,x_{n}\rangle = \langle (a,b),(0,0)\rangle = 0$ for $n\geq 3$
            \item $\langle x,y_{1}\rangle = \langle (a,b),(5i,0)\rangle = -5ia$
             \item $\langle x,y_{2}\rangle = \langle (a,b),(0,i)\rangle = -bi$
              \item $\langle x,y_{n}\rangle = \langle (a,b),(0,0)\rangle = 0$ for $n\geq 3$
        \end{itemize}
       From where,
       \begin{align*}
           &\bullet \sum\limits_{n \in \mathbb{N}}|\langle x, x_{n}\rangle |^2 = |-ai|^2+|-3bi|^2 = |a|^2+9|b|^2\leq 9(|a|^2+|b|^2) = 9 \Vert (a,b) \Vert^{2}\\
           &\bullet \sum\limits_{n \in \mathbb{N}}|\langle x, y_{n}\rangle |^2 = |-5ai|^2+|-bi|^2 = 25|a|^2+|b|^2\leq 25(|a|^2+|b|^2) = 25 \Vert (a,b) \Vert^{2}
       \end{align*}

That is, $\mathbb{X}$ and $\mathbb{Y}$ are Bessel sequences.\\

Now, let’s consider the sequences $\mathbb{F}=\{f_n\}_{n \in \mathbb{N}}$ and $\mathbb{G}=\{g_n\}_{n \in \mathbb{N}}$ given by $f_{1}= (i,0), f_{2}=(0,i),$ $f_{n}=(0,0),$ $g_{1}=((1/5)i,0), g_{2}=(0,3i),$ $g_{n}=(0,0)$ for $n\geq 3.$ Then,
\begin{itemize}
            \item $\langle x,f_{1}\rangle = \langle (a,b),(i,0)\rangle = -ai$
            \item $\langle x,f_{2}\rangle = \langle (a,b),(0,i)\rangle = -bi$
            \item $\langle x,f_{n}\rangle = \langle (a,b),(0,0)\rangle = 0$ for $n\geq 3$
            \item $\langle x,g_{1}\rangle = \langle (a,b),(1/5i,0)\rangle = (-1/5)ia$
             \item $\langle x,g_{2}\rangle = \langle (a,b),(0,3i)\rangle = -3bi$
              \item $\langle x,g_{n}\rangle = \langle (a,b),(0,0)\rangle = (0,0)$ for $n\geq 3$
        \end{itemize}
Therefore:
\begin{align*}
          &\bullet  \sum\limits_{n \in \mathbb{N}}|\langle x,f_{n}\rangle |^{2}= |-ai|^{2}+|-bi|^{2}= |a|^{2}+|b|^{2} = \Vert (a,b) \vert^{2}\\
          &\bullet  \sum\limits_{n \in \mathbb{N}}|\langle x,g_{n}\rangle |^{2}= |(-1/5)ai|^{2}+|-3bi|^{2}= (1/25)|a|^{2}+9|b|^{2} \leq 9 \Vert (a,b) \Vert^{2}
       \end{align*}

        Therefore, $\mathbb{F}$ and $\mathbb{G}$ are Bessel sequences.
Finally,

\begin{align*}
     \sum_{n \in \mathbb{N}} \left( \langle x, f_n\rangle x_n  + \langle x, g_n\rangle y_n \right) &= \langle x, f_1\rangle x_1  + \langle x, g_1\rangle y_1+\langle x, f_2\rangle x_2  + \langle x, g_2\rangle y_2\\
     &= -ai(i,0)+(-1/5ai)(5i,0)+(-bi)(0,3i)+(-3bi)(0,i)\\
     &=(a,0)+(a,0)+(0,3b)+(0,3b)=(2a,6b)\\
     &=T((a,b)).
\end{align*}

Based on all of the above, it follows that $(\mathbb{X},\mathbb{Y})$ is a bi-atomic system for $T.$
    \end{example}

    \begin{theorem}\label{thm:dualidad_biatomica}
    Let $(H, \langle \cdot, \cdot \rangle)$ be a Hilbert space,
    $T \in B(H)$, and $\mathbb{X} = \{x_n\}_{n \in \mathbb{N}}$,
    $\mathbb{Y} = \{y_n\}_{n \in \mathbb{N}}$ be sequences in $H$.
    If $(\mathbb{X}, \mathbb{Y})$ is a bi-atomic system for $T$ in $H$, with auxiliary sequences $\mathbb{F} = \{f_n\}_{n \in \mathbb{N}}$ and $\mathbb{G} = \{g_n\}_{n \in \mathbb{N}},$
    then $(\mathbb{F}, \mathbb{G})$ is a bi-atomic system for $T^{*}$ in $H.$
\end{theorem}

\begin{proof}
    Since $(\mathbb{X}, \mathbb{Y})$ is a bi-atomic system for $T$, there exist sequences of Bessel functions $\mathbb{F} = \{f_n\}_{n \in \mathbb{N}}$ 
    and $\mathbb{G} = \{g_n\}_{n \in \mathbb{N}}$  in $H$ such that
    \begin{equation}\label{eq:rep_T}
        T(x)
        = \sum_{n \in \mathbb{N}}
          \bigl(\langle x, f_n \rangle\, x_n
               + \langle x, g_n \rangle\, y_n\bigr),
        \qquad \forall\, x \in H.
    \end{equation}
    Furthermore, by Definition\eqref{sistema biatomico}, $\mathbb{X}$ and $\mathbb{Y}$ are Bessel sequences, respectively. We must verify the conditions of the Definition \eqref{sistema biatomico}
    for the pair $(\mathbb{F}, \mathbb{G})$ with respect to the operator $T^{*}$. Indeed,\\

By hypothesis, $\mathbb{F}$ and $\mathbb{G}$ are Bessel sequences. Furthermore,
  for any $z \in H$, we have
    \begin{align*}
        \langle T(x), z \rangle
        &= \left\langle
               \sum_{n \in \mathbb{N}}
               \bigl(\langle x, f_n \rangle\, x_n
                    + \langle x, g_n \rangle\, y_n\bigr),\,
               z
           \right\rangle
        = \sum_{n \in \mathbb{N}}
           \bigl(
               \langle x, f_n \rangle \langle x_n, z \rangle
             + \langle x, g_n \rangle \langle y_n, z \rangle
           \bigr) \\
           &=\sum_{n \in \mathbb{N}}
           \bigl(
               \langle x, f_n \rangle \overline{\langle z,x_n \rangle}
             + \langle x, g_n \rangle \langle y_n, z \rangle
           \bigr) \\
        &= \left\langle
               x,\;
               \sum_{n \in \mathbb{N}}
               \bigl(
                   \langle z, x_n \rangle\, f_n
                 + \langle z, y_n \rangle\, g_n
               \bigr)
           \right\rangle.
    \end{align*}

 Since $\langle T(x), z \rangle = \langle x, T^{*}(z) \rangle$
    for all $x \in H$, it follows that
    \begin{equation}\label{eq:rep_Tstar}
        T^{*}(z)
        = \sum_{n \in \mathbb{N}}
          \bigl(
              \langle z, x_n \rangle\, f_n
            + \langle z, y_n \rangle\, g_n
          \bigr),
        \qquad \forall\, z \in H.
    \end{equation}
    This is precisely the representation required by the
    Definition \eqref{sistema biatomico} for the pair $(\mathbb{F}, \mathbb{G})$ with respect to $T^{*},$
    with auxiliary sequences $\mathbb{X} = \{x_n\}$ and $\mathbb{Y} = \{y_n\}$,
    which are Bessel sequences.
    Therefore, $(\mathbb{F}, \mathbb{G})$ is a bi-atomic system for $T^{*}$
    in $H$.
\end{proof}
The following result presents an equivalent way of showing that a pair of sequences is a bi-atomic system for a bounded linear operator.
\begin{theorem}\label{equivalencia para biatomicos}
     Let $(H, \langle \cdot, \cdot \rangle)$ be a Hilbert space, $T\in \mathcal{B}({H})$, and $\mathbb{X}=\{x_n\}_{n \in \mathbb{N}}$, $\mathbb{Y}= \{y_n\}_{n \in \mathbb{N}}$ be two sequences in $H$. The following statements are equivalent:
 \begin{itemize}
     \item[$(i)$] $(\mathbb{X},\mathbb{Y})$ is a bi-atomic system for $T$ in $H.$
     \item[$(ii)$]
     \begin{itemize}
         \item[$(a)$] There exists a constant $C>0$ such that for every $x\in H,$ there exist sequences $a_x=\{a_n\}_{n\in \mathbb{N}},$ $ b_x=\{b_n\}_{n\in \mathbb{N}} \in \ell^2(\mathbb{N})$ with  $\ \Vert a_x \Vert_{\ell^2(\mathbb{N})} \leq C \Vert x \Vert$ and  $\Vert b_x \Vert_{\ell^2(\mathbb{N})} \leq C \Vert x \Vert$ such that
     $T(x) = \sum \limits_{n \in \mathbb{N}} (a_n x_n + b_n y_n).$
     \item[$(b)$] The series $\sum \limits_{n \in \mathbb{N}}(c_nx_n+d_ny_n)$ converges in $H$ for all $\{c_n\}_{n\in \mathbb{N}}, \{d_n\}_{n\in \mathbb{N}} \in \ell^{2}(\mathbb{N}).$
     \end{itemize}
 \end{itemize}    
 \end{theorem}
 \begin{proof}

    \begin{description}
       \item[$(i) \Rightarrow (ii)$] Suppose that $(\mathbb{X},\mathbb{Y})$ is a bi-atomic system for $T$ on $H$; then there exist $\mathbb{F}=\{f_n\}_{n\in \mathbb{N}},$\ $\mathbb{G}=\{g_n\}_{n\in \mathbb{N}}$ Bessel sequences with bounds $\beta_{f}$ and $\beta_{g}$ such that for every $x\in H$, we have $T(x)=\sum \limits_{n \in \mathbb{N}} \left( \langle x, f_n\rangle x_n + \langle x, g_n\rangle y_n  \right).$ Let $a_{x}=\{\langle x, f_n\rangle\}_{n\in \mathbb{N}},$ where $a_{n}(x)=\langle x, f_n\rangle$ and $b_{x}=\{\langle x, g_n\rangle\}_{n\in \mathbb{N}},$ where $b_{n}(x)=\langle x, g_n\rangle,$ then
       \begin{align*}
           &\bullet  \Vert \{\langle x,f_n\rangle\} \Vert_{\ell^{2}(\mathbb{N})}^2 = \sum_{n \in \mathbb{N}}|\langle x,f_n\rangle|^2 \leq \beta_{f} \Vert x \Vert^{2}\Longrightarrow \Vert a_x \Vert_{\ell^{2}(\mathbb{N})} \leq \sqrt{\beta_f}\Vert x \Vert\\
           &\bullet  \Vert \{\langle x,g_n\rangle\} \Vert_{\ell^{2}(\mathbb{N})}^2 = \sum_{n \in \mathbb{N}}|\langle x,g_n\rangle|^2 \leq \beta_{g} \Vert x \Vert^{2}\Longrightarrow \Vert b_x \Vert_{\ell^{2}(\mathbb{N})} \leq \sqrt{\beta_g}\Vert x \Vert.
       \end{align*}
    Let $C = \max\{\sqrt{\beta_f}, \sqrt{\beta_g}\}$. It follows that  $\Vert a_x \Vert_{\ell^{2}(\mathbb{N})} \leq C\Vert x \Vert$  and \\ $\Vert b_x \Vert_{\ell^{2}(\mathbb{N})} \leq C\Vert x \Vert$. Furthermore,
    $$\sum _{n \in \mathbb{N}}(a_nx_n +b_ny_n) = \sum_{n \in \mathbb{N}} \left( \langle x, f_n\rangle x_n  + \langle x, g_n\rangle y_n \right) = T(x),$$ This proves $(a).$
    Now, let's prove $(b).$ Indeed, let $\{c_n\}_{n \in \mathbb{N}},$ $\{d_n\}_{n \in \mathbb{N}} \in \ell^{2}(\mathbb{N})$, then for $N > M$ and since $\mathbb{X}$ is a sequence of Bessel functions with bound $\beta$, we have that:
    \begin{align*}
        \left \Vert \sum_{n=1}^{N}c_nx_n - \sum_{n=1}^{M}c_nx_n \right \Vert &= \left \Vert \sum_{n=M+1}^{N}c_nx_n \right \Vert = \substack{\sup\\ \Vert y \Vert =1 } \left| \left\langle\sum_{n=M+1}^{N} c_nx_n, y \right\rangle  \right|\\
        &=\substack{\sup\\ \Vert y \Vert =1 } \left| \sum_{n=M+1}^{N}\left\langle c_nx_n, y \right\rangle  \right|\\
        &\leq \substack{\sup\\ \Vert y \Vert =1 }\left\{ \left(\sum_{n=M+1}^{N}|c_n|^2 \right)^{1/2} \left(\sum_{n=M+1}^{N}|\langle x_n,y|^2 \right)^{1/2}\right\}\\
        &\leq \left(\sum_{n=M+1}^{N}|c_n|^2 \right)^{1/2}  \substack{\sup\\ \Vert y \Vert =1 }\left\{ \left(\sum_{n=M+1}^{N}|\langle x_n,y|^2 \right)^{1/2}\right\}\\
        &\leq \left(\sum_{n=M+1}^{N}|c_n|^2 \right)^{1/2} \sqrt{\beta}
    \end{align*}
    Since $\{c_n\}_{n \in \mathbb{N}} \in \ell^{2}(\mathbb{N}),$ then $\sum \limits_{n \in \mathbb{N}}c_nx_n $ is a Cauchy sequence in $H$ and therefore converges. Similarly, $\sum\limits_{n \in \mathbb{N}}d_ny_n$ converges; therefore
    \vspace{-10pt}
    $$\sum _{n \in \mathbb{N}}c_nx_n + \sum _{n \in \mathbb{N}}d_ny_n  = \sum _{n \in \mathbb{N}}(c_nx_n +d_ny_n)$$ 

\vspace{-10pt}
    and thus the series $\sum _{n \in \mathbb{N}}(c_nx_n + d_ny_n)$ converges in $H.$\\
    
    \item[$(ii) \Rightarrow (i)$] Suppose that $(a)$ and $(b)$ hold; let us see that $(\mathbb{X}, \mathbb{Y})$ is a bi-atomic system for $T$. For $k \in \mathbb{N}$, consider the family of operators\\ $W_k: \ell^{2}(\mathbb{N}) \times \ell^{2}(\mathbb{N}) \longrightarrow H $ given by 
    \vspace{-5pt}
    $$W_k\left(\left(\{c_n\}_{n \in \mathbb{N}} , \{d_n\}_{n \in \mathbb{N}}\right) \right) =  \sum \limits _{n=1}^{k}(c_nx_n + d_ny_n). $$ Let us show that the operators $W_k$ are linear and bounded. Indeed.
    
Let $\left(\{c_n\}_{n \in \mathbb{N}} , \{d_n\}_{n \in \mathbb{N}}\right),$ $\left(\{k_n\}_{n \in \mathbb{N}} , \{p_n\}_{n \in \mathbb{N}}\right) \in \ell^{2}(\mathbb{N}) \times \ell^{2}(\mathbb{N})$ and $\lambda \in \mathbb{C},$ then
        \begin{align*}
            &W_k\left(\left(\{c_n\}_{n \in \mathbb{N}} , \{d_n\}_{n \in \mathbb{N}}\right) + \lambda\left(\{k_n\}_{n \in \mathbb{N}} , \{p_n\}_{n \in \mathbb{N}}\right) \right)\\
           &=W_k\left(\left(\{c_n\}_{n \in \mathbb{N}}+ \lambda \{k_n\}_{n \in \mathbb{N}} , \{d_n\}_{n \in \mathbb{N}} + \lambda\{p_n\}_{n \in \mathbb{N}}\right) \right)\\
           &=\sum \limits _{n=1}^{k}(( c_n +\lambda k_n)x_n + (\alpha d_n +\lambda p_n)y_n)\\
           &=  \sum \limits _{n=1}^{k} (c_nx_n+d_ny_n)+ \lambda \sum \limits _{n=1}^{k}(k_nx_n+p_ny_n)\\
           &= W_k\left(\left(\{c_n\}_{n \in \mathbb{N}} , \{d_n\}_{n \in \mathbb{N}}\right) \right) + \lambda W_k\left(\left(\{k_n\}_{n \in \mathbb{N}} , \{p_n\}_{n \in \mathbb{N}}\right) \right)   
        \end{align*}
    Therefore, the $W_k$ are linear. Furthermore,

    \begin{align*}
        \|W_k\left(\left(\{c_n\}_{n \in \mathbb{N}} , \{d_n\}_{n \in \mathbb{N}}\right) \right)\| &= \left\|\sum \limits _{n=1}^{k}(c_nx_n + d_ny_n)\right\| \leq \left\|\sum \limits \limits _{n=1}^{k}c_nx_n \right\| + \left\|\sum \limits _{n=1}^{k} d_ny_n\right\|.
    \end{align*}

    But,

    \begin{align*}
      \left\|\sum \limits _{n=1}^{k}c_nx_n \right\| &= \sup_{\|y\|=1}\left|\langle \sum \limits _{n=1}^{k}c_{n}x_{n} , y\rangle  \right|  = \sup_{\|y\|=1}\left|\sum \limits _{n=1}^{k}c_{n}\langle x_{n} , y\rangle  \right|\\
      &\leq \sup_{\|y\|=1} \left\{\left(\sum \limits _{n=1}^{k}|c_{n}|^2 \right)^{1/2} \left(\sum \limits _{n=1}^{k}|\langle x_{n} , y\rangle|^2 \right)^{1/2}\right\}\\
      &\leq \left(\sum \limits _{n=1}^{k}|c_{n}|^2 \right)^{1/2} \sup_{\|y\|=1} \left\{ \left(\sum \limits _{n=1}^{k}|\langle x_{n} , y\rangle|^2 \right)^{1/2}\right\}\\
      &\leq \left(\sum \limits _{n=1}^{k}|c_{n}|^2 \right)^{1/2} \sup_{\|y\|=1} \left\{ \left(\sum \limits _{n=1}^{k}\|y\|^2 \|x_n\|^2 \right)^{1/2}\right\}\\
      &\leq \left(\sum \limits _{n=1}^{k}|c_{n}|^2 \right)^{1/2}  \left(\sum \limits _{n=1}^{k} \|x_n\|^2 \right)^{1/2}
      = \|\{c_n\}\|_{\ell^{2}(\mathbb{N})} \left(\sum \limits _{n=1}^{k} \|x_n\|^2 \right)^{1/2}
    \end{align*}

    Now, $\left(\sum \limits _{n=1}^{k} \|x_n\|^2 \right)^{1/2}$ is finite for all $k\in \mathbb{N},$ so we can write\\ $M_k^{\mathbb{X}} :=\left (\sum \limits _{n=1}^{k} \|x_n\|^2 \right)^{1/2}$, and thus we have that $\left\|\sum \limits _{n=1}^{k}c_nx_n \right\| \leq M_k^{\mathbb{X}}\|\{c_n\}\|_{\ell^{2}(\mathbb{N})}.
    $ Similarly, it follows that  $\left\|\sum \limits _{n=1}^{k}d_ny_n \right\| \leq M_k^{\mathbb{Y}}\|\{d_n\}\|_ {\ell^{2}(\mathbb{N})},$ so that \\
    $\left\|\sum \limits \limits _{n=1}^{k}c_nx_n \right\| + \left\|\sum \limits _{n=1}^{k} d_ny_n\right\| \leq \max 
    \{M_k^{\mathbb{X}}, M_k^{\mathbb{Y}}\}  \left(\|\{c_n\}\|_{\ell^{2}(\mathbb{N})} + \|\{d_n\}\|_{\ell^{2}(\mathbb{N})} \right)$\\
    
    This ensures that\\
    
    $\|W_k\left(\left(\{c_n\}_{n \in \mathbb{N}} , \{d_n\}_{n \in \mathbb{N}}\right) \right)\| \leq \max \{M_k^{\mathbb{X}},M_k^{\mathbb{Y}}\}  \left(\|\{c_n\}\|_{\ell^{2}(\mathbb{N})} + \|\{d_n\}\|_{\ell^{2}(\mathbb{N})} \right).$ That is, the $W_k$ are bounded.
     In turn, the limit, $$  \lim \limits_{k\to \infty} W_k\left(\{c_n\}_{n \in \mathbb{N}} , \{d_n\}_{n \in \mathbb{N}} \right) = \sum \limits _{n \in \mathbb{N}}(c_nx_n + d_ny_n)$$
    exists in $H$, and we will denote it by $W\left(\{c_n\}_{n \in \mathbb{N}} , \{d_n\}_{n \in \mathbb{N}} \right).$  Thus, since the operators $W_k$ are bounded and converge pointwise to $W,$ then by the Banach–Steinhaus Theorem \eqref{Teorema de Banach-Steinhaus} It follows that $W$ is linear and bounded.
    
Now, according to the Proposition \eqref{prop:adjunto}, exist $W^*:H\to\ell^{2}(\mathbb{N})\times\ell^{2}(\mathbb{N})$ bounded with $\|W^*\|=\|W\|$ such that for all $(\{c_n\}_{n \in \mathbb{N}},\{d_n\}_{n \in \mathbb{N}}) \in \ell^{2}(\mathbb{N})\times\ell^{2}(\mathbb{N})$ and all $h \in H$
\begin{equation*}\label{eq:adj}
\big\langle W\big((\{c_n\}_{n \in \mathbb{N}},\{d_n\}_{n \in \mathbb{N}})\big),h\big\rangle=\big\langle(\{c_n\}_{n \in \mathbb{N}},\{d_n\}_{n \in \mathbb{N}}),W^*h\big\rangle.
\end{equation*}
Next, let's write $W^*(h)=(\{u_n\}_{n \in \mathbb{N}},\{v_n\}_{n \in \mathbb{N}})$ con $(\{u_n\}_{n \in \mathbb{N}},\{v_n\}_{n \in \mathbb{N}}) \in \ell^{2}(\mathbb{N})\times\ell^{2}(\mathbb{N})$. For $m\in\mathbb{N}$ let $e_m=\{\delta_{mn}\}_{n\in\mathbb{N}}$, where $\delta_{mn}=1$ if $m=n$ and $\delta_{mn}=0$ if $m\neq n$ (the sequence with a $1$ in position $m$ and zeros in the other positions). Then $W\big((\{\delta_{mn}\}_{n\in\mathbb{N}},\{0\}_{n\in\mathbb{N}})\big)=x_m$ and by the Definition \eqref{norma l2xl2} we have
\begin{align*}
    \langle x_m,h\rangle &=\langle W\big((\{\delta_{mn}\}_{n\in\mathbb{N}},\{0\}_{n\in\mathbb{N}})\big),h\rangle\\
    &=\langle(\{\delta_{mn}\}_{n\in\mathbb{N}},\{0\}_{n\in\mathbb{N}}),(\{u_n\}_{n \in \mathbb{N}},\{v_n\}_{n \in \mathbb{N}})\rangle\\
    &=\langle e_m,\{u_n\}_{n\in\mathbb{N}}\rangle_{\ell^2(\mathbb{N})}=\overline{u_m}.
\end{align*}
that is, $u_m=\langle h,x_m\rangle$. Similarly, using $W\big((\{0\}_{n\in\mathbb{N}},\{\delta_{mn}\}_{n\in\mathbb{N}})\big)=y_m$, we obtain $v_m=\langle h,y_m\rangle$. Thus,
\[
W^*(h)=\big(\{\langle h,x_n\rangle\}_{n\in\mathbb{N}},\{\langle h,y_n\rangle\}_{n\in\mathbb{N}}\big),
\]
and by the Definition \eqref{norma l2xl2},
\[
\sum_{n\in\mathbb{N}}|\langle h,x_n\rangle|^2+\sum_{n\in\mathbb{N}}|\langle h,y_n\rangle|^2=\|W^*(h)\|^2\le\|W\|^2\|h\|^2 .
\]
Each of the two sums is less than or equal to $\|W\|^2\|h\|^2$, so $\mathbb{X}$ and $\mathbb{Y}$ are Bessel sequences with bound $\|W\|^2$.
 
\smallskip
 On the other hand, by  (a), there exist $a_x=\{a_n\}_{n\in\mathbb{N}}, b_x=\{b_n\}_{n\in\mathbb{N}}\in\ell^{2}(\mathbb{N})$ such that the series $\sum\limits_ {n\in\mathbb{N}}(a_nx_n+b_ny_n)$ converges to $T(x)$. Its partial sums are $W_k\big((\{a_n\}_{n\in\mathbb{N}},\{b_n\}_{n\in\mathbb{N}})\big)$, so,
\[
T(x)=\lim_{k\to\infty}W_k\big((\{a_n\}_{n\in\mathbb{N}},\{b_n\}_{n\in\mathbb{N}})\big)\in\mathcal{R}(W).
\]
Therefore, $\mathcal{R}(T)\subset\mathcal{R}(W)$. We apply Douglas's Theorem \eqref{thm: Douglas} with \\ $H_1=H$, $H_2=\ell^{2}(\mathbb{N})\times\ell^{2}(\mathbb{N})$, $T_1=T\in B(H)$ and $T_2=W\in B(H_2,H),$ therefore, it exists $D\in B(H,\ell^{2}(\mathbb{N})\times\ell^{2}(\mathbb{N}))$, linear and bounded, such that
\[
T=WD .
\]

For each $x\in H,$ $T(x)=W(D(x)),$ so $D(x) \in \ell^{2}(\mathbb{N})\times\ell^{2}(\mathbb{N})$; we define $A_n(x)$ as the $n$th term of the first sequence and $B_n(x)$ the $n$th term of the second sequence, that is,
\[
D(x)=\big(\{A_n(x)\}_{n\in\mathbb{N}},\{B_n(x)\}_{n\in\mathbb{N}}\big).
\]
 
Since $D$ is linear, $A_n$ and $B_n$ are linear functionals (compositions of $D$ with a linear functional). Furthermore,
\[
|A_n(x)|\le\|Dx\|\le\|D\|\,\|x\|,\qquad |B_n(x)|\le\|Dx\|\le\|D\|\,\|x\|,
\]
then $A_n$ and $B_n$ are bounded. By Riesz's Representation Theorem \eqref{Teorema-representacion-Riesz}, there exist unique $f_n, g_n \in H$ such that for every$x\in H$
\[
A_n(x)=\langle x,f_n\rangle,\qquad B_n(x)=\langle x,g_n\rangle.
\]
Let $\mathbb{F}=\{f_n\}_{n\in\mathbb{N}}$ and $\mathbb{G}=\{g_n\}_{n\in\mathbb{N}},$ thus  $D(x)=(\{\langle x,f_n\rangle\},\{\langle x,g_n\rangle\})$ and by the Definition \eqref{norma  l2xl2}
\[
\sum_{n\in\mathbb{N}}|\langle x,f_n\rangle|^2+\sum_{n\in\mathbb{N}}|\langle x,g_n\rangle|^2=\|Dx\|^2\le\|D\|^2\|x\|^2 .
\]
Therefore, $F$ and $G$ are Bessel sequences with bound $\|D\|^2$. Finally, by the definition of $W$,
\begin{align*}
    T(x)&=W(Dx)=\lim_{k\to\infty}\sum_{n=1}^{k}\big(\langle x,f_n\rangle x_n+\langle x,g_n\rangle y_n\big)\\
    &=\sum_{n\in\mathbb{N}}\big(\langle x,f_n\rangle x_n+\langle x,g_n\rangle y_n\big), \quad\forall x\in H .
\end{align*} 
 which was all that was needed to conclude that $(\mathbb{X},\mathbb{Y})$ is a bi-atomic system for $T$.
    \end{description} 
 \end{proof}

 \begin{theorem}\label{thm:proyector}
Let $(H, \langle\cdot,\cdot\rangle)$ be a Hilbert space and $P \in B(H)$ an orthogonal projector.
If $(\{x_n\}_{n \in \mathbb{N}}, \{y_n\}_{n \in \mathbb{N}})$ is a bi-atomic system for $T \in B(H)$
in $H$, then
\begin{itemize}
    \item[$(i)$] $(\{Px_n\}_{n \in \mathbb{N}}, \{Py_n\}_{n \in \mathbb{N}})$ is a bi-atomic system
for $PT$ in $PH$
\item[$(ii)$] $(\{(I-P)x_n\}_{n \in \mathbb{N}}, \{(I-P)z_n\}_{n \in \mathbb{N}})$ is a
bi-atomic system for $(I-P)T$ in $(I-P)H$.
\end{itemize}
\end{theorem}

\begin{proof}
  \begin{description}
       \item[$(i)$] Since $(\mathbb{X}=\{x_n\}_{n \in \mathbb{N}},\mathbb{Y}= \{y_n\}_{n \in \mathbb{N}})$ is a bi-atomic system for $T$, there exist Bessel sequences
$\{f_n\}$ (with bound $B_f$) and $\{g_n\}$ (with bound $B_g$) in $H$ such that:
\[
  T(x) = \sum_{n \in \mathbb{N}}\bigl(\langle x, f_n\rangle x_n + \langle x, g_n\rangle y_n\bigr),
  \ \forall\, x \in H.   \tag{1}
\] 
Let $z \in PH$; then there exists $x \in H$ such that $z = P(x)$, so\\
$P(z) = P(P(x)) = P^2(x) = P(x) = z$. Therefore:
\[
  \sum_{n \in \mathbb{N}}|\langle z, Px_n\rangle|^2
  = \sum_{n \in \mathbb{N}}|\langle Pz, x_n\rangle|^2
  = \sum_{n \in \mathbb{N}}|\langle z, x_n\rangle|^2
  \leq B_x\|z\|^2,
\]
Then $\{Px_n\}_{n \in \mathbb{N}}$ is a Bessel sequence in $PH$ with norm $B_x$. Similarly, $\{Py_n\}_{n \in \mathbb{N}}$ is a Bessel sequence in $PH$
with norm $B_y$. Similarly, it can be shown that $\{Pf_n\}_{n \in \mathbb{N}}$ and $\{Pg_n\}_{n \in \mathbb{N}}$ are Bessel sequences with bounds $B_f$ and $B_g$. Now,
let $z \in PH$; then there exists $x \in H$ such that $z = Px$. Applying $P$ to $(1)$ with $Px$ in place of
$x$, and using the fact that $P^2 = P$ and $P^* = P$, we obtain:
\begin{align*}
  PT(z) &= PT(P(x))\\
  &= P\left( \sum_{n \in \mathbb{N}}\bigl(\langle Px, f_n\rangle x_n
                                  + \langle Px, g_n\rangle y_n\bigr) \right) \\
     &= \sum_{n \in \mathbb{N}}\bigl(\langle Px, f_n\rangle Px_n
                                 + \langle Px, g_n\rangle Py_n\bigr) \\
      &= \sum_{n \in \mathbb{N}}\bigl(\langle P^2x, f_n\rangle Px_n
                                 + \langle P^2x, g_n\rangle Py_n\bigr)\\
      &= \sum_{n \in \mathbb{N}}\bigl(\langle Px, Pf_n\rangle Px_n
                                 + \langle Px, Pg_n\rangle Py_n\bigr) \\
      &= \sum_{n \in \mathbb{N}}\bigl(\langle z, Pf_n\rangle Px_n
                                 + \langle z, Pg_n\rangle Py_n\bigr).
\end{align*}
Therefore, $(\{Px_n\}_{n \in \mathbb{N}}, \{Py_n\}_{n \in \mathbb{N}})$ is a bi-atomic system for $PT$ in $PH.$

       \item[$(ii)$] Since $I - P$ is also an orthogonal projector, the same argument applied with $I - P$ in
instead of $P$ proves that $(\{(I-P)x_n\}_{n \in \mathbb{N}}, \{(I-P)y_n\}_{n \in \mathbb{N}})$ is a bi-atomic system for $(I-P)T$ in $(I-P)H$.
\end{description}
\end{proof}

\begin{theorem}\label{thm:unitario}
Let $(H, \langle\cdot,\cdot\rangle)$ be a Hilbert space, $(\{x_n\}_{n \in \mathbb{N}},
\{y_n\}_{n \in \mathbb{N}})$ a bi-atomic system for $T \in B(H)$ on $H$, and $L \in B(H)$
a unitary operator that commutes with $T$, that is, $LT = TL$. Then
$(\{Lx_n\}_{n \in \mathbb{N}}, \{Ly_n\}_{n \in \mathbb{N}})$ is a bi-atomic system for $T$ on $H$.
\end{theorem}

\begin{proof}
Since $(\{x_n\}_{n \in \mathbb{N}}, \{y_n\}_{n \in \mathbb{N}})$ is a bi-atomic system for $T$, there exist Bessel sequences
$\{f_n\}_{n \in \mathbb{N}}$ (with bound $B_f$) and $\{g_n\}_{n \in \mathbb{N}}$ (with bound $B_g$) in $H$ such that
\[
  T(x) = \sum_{n \in \mathbb{N}}\bigl(\langle x, f_n\rangle x_n + \langle x, g_n\rangle y_n\bigr)
  \qquad \forall\, x \in H. \tag{2}
\]
Since $L$ is unitary, $L^*L = LL^* = I$, and in particular $\|L^*x\| = \|x\|$ for all
$x \in H$, then
\[
  \sum_{n \in \mathbb{N}}|\langle x, Lx_n\rangle|^2
  = \sum_{n \in \mathbb{N}}|\langle L^*x, x_n\rangle|^2
  \leq B_x\|L^*x\|^2
  = B_x\|x\|^2,
\]
Then
$\{Lx_n\}_{n \in \mathbb{N}}$ is a Bessel sequence with modulus $B_x$. Similarly, $\{Ly_n\}_{n \in \mathbb{N}}$ is a Bessel sequence with modulus $B_y$. Similarly, it can be shown that $\{Lf_n\}_{n \in \mathbb{N}}$ and $\{Lg_n\}_{n \in \mathbb{N}}$ are sequences of Bessel functions with bounds $B_f$ and $B_g$.
Since $LT = TL$ and $L$ is unitary ($L^{-1} = L^*$), for all $x \in H$:
\[
  T(x) = T(LL^*(x)) = (TL)L^*(x) = (LT)L^*(x) = L(TL^*(x)).
\]
Applying $(2)$ with $L^*(x)$ in place of $x$:
\[
  TL^*(x) = \sum_{n \in \mathbb{N}}\bigl(\langle L^*(x), f_n\rangle x_n
                                  + \langle L^*(x), g_n\rangle y_n\bigr).
\]
Next,
\begin{align*}
  T(x) = L(TL^*(x))
     &= L\left(\sum_{n \in \mathbb{N}}\bigl(\langle L^*(x), f_n\rangle x_n
                                 + \langle L^*(x), g_n\rangle z_n\bigr)\right) \\
    & = \sum_{n \in \mathbb{N}}\bigl(\langle L^*(x), f_n\rangle Lx_n
                                + \langle L^*(x), g_n\rangle Ly_n\bigr)\\
     &= \sum_{n \in \mathbb{N}}\bigl(\langle x, Lf_n\rangle Lx_n
                               + \langle x, Lg_n\rangle Ly_n\bigr),
  \qquad \forall\, x \in H.
\end{align*}

Therefore, $(\{Lx_n\}, \{Lz_n\})$ is a bi-atomic system for $T$ in $H.$
\end{proof}
The following proposition establishes a connection between Bessel biframes and bi-atomic systems.
\begin{proposition}\label{bimarco bessel implica biatomico}
   Let $(H, \langle\cdot,\cdot\rangle)$ be a Hilbert space, $\mathbb{X}=\{x_n\}_{n \in \mathbb{N}},
\mathbb{Y}=\{y_n\}_{n \in \mathbb{N}}$ Bessel sequences in $\mathcal{H},$ with bounds $\beta_{\mathbb{X}}$ and $\beta_{\mathbb{Y}}$, respectively. If $(\mathbb{X}, \mathbb{Y})$ is a bimarco for $\mathcal{H},$ then $(\mathbb{X}, \mathbb{Y})$ is a bi-atomic system for the identity operator $I$ of $\mathcal{H}.$
\end{proposition}

\begin{proof}
    Since $(\mathbb{X}, \mathbb{Y})$ is a biframe for $\mathcal{H},$ then for every $x\in \mathcal{H}$ there exist $0<\lambda \leq \beta < \infty$ such that
    \begin{equation}
\lambda \|x\|^{2} \le
\sum \limits_{n \in \mathbb{N}} \langle x,x_n\rangle \,\langle y_n,x\rangle 
\le \beta \|x\|^{2}.
    \end{equation}
    In addition, there is the biframe operator $S_{\mathbb{X}\mathbb{Y}}(x):= \sum\limits_{n \in \mathbb{N}} \langle x,x_{n}\rangle y_{n},$ which is linear, bounded, and invertible—that is, $S_{\mathbb{X}\mathbb{Y}}^{-1}$ exists. Now, let’s define the sequences\\ $\mathbb{F}= \{f_n\}_{n \in \mathbb{N}} =\{\frac{1}{2}S_{\mathbb{X}\mathbb{Y}}^{-1}(y_n)\}_{n \in \mathbb{N}}$ and $\mathbb{G}= \{g_n\}_{n \in \mathbb{N}} =\{\frac{1}{2}S_{\mathbb{Y}\mathbb{X}}^{-1}(x_n)\}_{n \in \mathbb{N}},$ then
    \begin{align*}
        \sum \limits_{n\in \mathbb{N}}\left|\langle x, f_n\rangle \right|^2 &= \sum \limits_{n\in \mathbb{N}}\left|\left\langle x, \frac{1}{2}S_{\mathbb{X}\mathbb{Y}}^{-1}(y_n)\right\rangle \right|^2 = \frac{1}{4}\sum \limits_{n\in \mathbb{N}}\left|\left\langle S_{\mathbb{Y}\mathbb{X}}^{-1} (x), y_n\right\rangle \right|^2\\
        &\leq \dfrac{1}{4} \beta_{\mathbb{Y}}\Vert S_{\mathbb{Y}\mathbb{X}}^{-1} (x) \Vert^2 \leq \beta_{\mathbb{Y}}C\Vert x \Vert^2
    \end{align*}
    This implies that $\mathbb{F}$ is a Bessel sequence. Similarly, it can be shown that $\mathbb{G}$ is also a Bessel sequence.\\
    On the other hand,
    \begin{align*}
        \sum \limits_{n\in \mathbb{N}} \left( \langle x, f_{n} \rangle x_{n} + \langle x, g_n \rangle y_{n}\right)& = \sum \limits_{n\in \mathbb{N}} \left( \left\langle x, \frac{1}{2}S_{\mathbb{X}\mathbb{Y}}^{-1}(y_n) \right\rangle x_{n} + \left\langle x, \frac{1}{2}S_{\mathbb{Y}\mathbb{X}}^{-1}(x_n) \right\rangle y_{n}\right) \\
        &= \frac{1}{2} \sum \limits_{n\in \mathbb{N}} \left( \left\langle S_{\mathbb{Y}\mathbb{X}}^{-1} (x), y_n \right\rangle x_{n} + \left\langle S_{\mathbb{X}\mathbb{Y}}^{-1} (x), x_n \right\rangle y_{n}\right)\\
        &=\frac{1}{2}\left( \sum \limits_{n\in \mathbb{N}}\left\langle S_{\mathbb{Y}\mathbb{X}}^{-1} (x), y_n \right\rangle x_{n} +  \sum \limits_{n\in \mathbb{N}} \left\langle S_{\mathbb{X}\mathbb{Y}}^{-1} (x), x_n \right\rangle y_{n}  \right)\\
        &= \frac{1}{2}\left(S_{\mathbb{Y}\mathbb{X}}(S_{\mathbb{Y}\mathbb{X}}^{-1}(x)) + S_{\mathbb{X}\mathbb{Y}}(S_{\mathbb{X}\mathbb{Y}}^{-1}(x)) \right)
         =\frac{1}{2}(x+x)\\
        &=I(x)
    \end{align*}
    Therefore, $(\mathbb{X}, \mathbb{Y})$ is a bi-atomic system for $I.$
\end{proof}

\begin{theorem}\label{thm:concatenacion_biatomica}
    Let $(H, \langle \cdot, \cdot \rangle)$ be a Hilbert space,
    $T, W \in B(H)$, and the sequences
    $\mathbb{X} = \{x_n\}_{n \in \mathbb{N}}$,
    $\mathbb{Y} = \{y_n\}_{n \in \mathbb{N}}$,
    $\mathbb{U} = \{u_n\}_{n \in \mathbb{N}}$,
    $\mathbb{V} = \{v_n\}_{n \in \mathbb{N}} \subset H$.
    If $(\mathbb{X}, \mathbb{Y})$ is a bi-atomic system for $T$ in $H$ and
    $(\mathbb{U}, \mathbb{V})$ is a bi-atomic system for $W$ in $H$,
    then
    \begin{equation*}
        (\widetilde{\mathbb{X}} = \{x_1, u_1, x_2, u_2, \ldots\}, \
        \widetilde{\mathbb{Y}} = \{y_1, v_1, y_2, v_2, \ldots\})
    \end{equation*}
    is a bi-atomic system for $(T + W)$ in $H$.
\end{theorem}

\begin{proof}
    Since $(\mathbb{X}, \mathbb{Y})$ is a bi-atomic system for $T$ and $(\mathbb{U}, \mathbb{V})$ is a bi-atomic system for $W$,
    there exist Bessel sequences $\mathbb{F} = \{f_n\}_{n \in \mathbb{N}}$
    (with bound $\beta_\mathbb{N},$)  $\mathbb{G} = \{g_n\}_{n \in \mathbb{N}}$
    (with bound $\beta_\mathbb{G}$), $\mathbb{P} = \{p_n\}_{n \in \mathbb{N}}$
    (with bound $\beta_\mathbb{P}$) and $\mathbb{K} = \{k_n\}_{n \in \mathbb{N}}$
    (with bound $\beta_\mathbb{K}$) in $H$ such that for every $x \in H$
    \begin{equation}\label{eq:rep_T_concat}
        T(x) = \sum_{n \in \mathbb{N}}
        \bigl(\langle x, f_n \rangle\, x_n
        + \langle x, g_n \rangle\, y_n\bigr).
    \end{equation}

    \begin{equation}\label{eq:rep_W_concat}
        W(x) = \sum_{n \in \mathbb{N}}
        \bigl(\langle x, p_n \rangle\, u_n
        + \langle x, k_n \rangle\, v_n\bigr).
    \end{equation}
    Furthermore, $\mathbb{X}$, $\mathbb{Y}$, $\mathbb{U}$, and $\mathbb{V}$ are Bessel sequences
    with coefficients $\beta_\mathbb{X}$, $\beta_\mathbb{Y}$, $\beta_\mathbb{U}$, and $\beta_\mathbb{V}$
    respectively, by condition $(1)$ of the
    Definition \eqref{sistema biatomico}.
    Let us verify the conditions of Definition \eqref{sistema biatomico} for $(\widetilde{\mathbb{X}}, \widetilde{\mathbb{Y}})$ with respect to the
    operator $(T + W)$. Indeed, let us first see that $(T + W) \in B(H)$, since for every
    $x \in H$:
    \[
        \|(T + W)(x)\|
        \leq \|T(x)\| + \|W(x)\|
        \leq M\|x\| + K\|x\| \leq \bigl(M + K\bigr)\|x\|.
    \]
 Therefore, $(T+W)$ is bounded. 
    \medskip
    
    Let us now show that $\widetilde{\mathbb{X}}$ and $\widetilde{\mathbb{Y}}$ are Bessel sequences.
    For every $x \in H$:
    \begin{align*}
        \sum_{n \in \mathbb{N}}
        |\langle x, \widetilde{x}_n \rangle|^2
        &= \sum_{n \in \mathbb{N}} |\langle x, x_n \rangle|^2
        + \sum_{n \in \mathbb{N}} |\langle x, u_n \rangle|^2 \\
        &\leq \beta_\mathbb{X} \|x\|^2 + \beta_\mathbb{U} \|x\|^2
        = (\beta_\mathbb{X} + \beta_\mathbb{U})\|x\|^2,
    \end{align*}
    Then $\widetilde{\mathbb{X}}$ is a Bessel sequence with bound
    $(\beta_\mathbb{X} + \beta_\mathbb{U})$.
    Similarly, $\widetilde{\mathbb{X}}$ is a Bessel sequence
    with bound $(\beta_\mathbb{Y} + \beta_\mathbb{V}).$
 
    \medskip

    Next, let's define the auxiliary sequences   \[
        \widetilde{\mathbb{F}} = \{f_1, p_1, f_2, p_2, \ldots\},
        \qquad
        \widetilde{\mathbb{G}} = \{g_1, k_1, g_2, k_2, \ldots\}.
    \]
    By the same argument as above,
    $\widetilde{\mathbb{F}}$ is a Bessel sequence with bound
    $\beta_\mathbb{F} + \beta_\mathbb{P}$ and $\widetilde{\mathbb{G}}$ is a Bessel sequence
    with bound $\beta_\mathbb{G} + \beta_\mathbb{K}$.
    On the other hand, for every $x \in H$:
    \begin{align*}
        \sum_{n \in \mathbb{N}}
        \bigl(
        \langle x, \widetilde{f}_n \rangle\, \widetilde{x}_n
        + \langle x, \widetilde{g}_n \rangle\, \widetilde{y}_n
        \bigr)
        &= \sum_{n \in \mathbb{N}}
        \bigl(\langle x, f_n \rangle\, x_n
        + \langle x, g_n \rangle\, y_n\bigr) 
         +\\
        &\sum_{n \in \mathbb{N}}
        \bigl(\langle x, p_n \rangle\, u_n
        + \langle x, q_n \rangle\, v_n\bigr) \\
        &= T(x) + W(x)
        = (T + W)(x),
    \end{align*}
    where the separation into two series is justified
    because each one converges to $H,$
    the first by \eqref{eq:rep_T_concat}
    and the second by \eqref{eq:rep_W_concat}.
 
    \medskip
    Therefore, $(\widetilde{\mathbb{X}}, \widetilde{\mathbb{Y}})$ is a
    bi-atomic system for $(T + W)$ in $H.$
\end{proof}

\begin{corollary}\label{cor:concat_bimarco}
    Under the assumptions of the Theorem \eqref{thm:concatenacion_biatomica} and 
    if, in addition, $(\mathbb{X}, \mathbb{Y})$ is a Bessel biframe for $H$ with bounds
    $\lambda_1 \leq \beta_1$ y $(\mathbb{U}, \mathbb{V})$ is a Bessel biframe
    for $H$ with bounds $\lambda_2 \leq \beta_2$, then
    $(\widetilde{\mathbb{X}}, \widetilde{\mathbb{Y}})$ is a Bessel biframe
    for $H$ with bounds $(\lambda_1 + \lambda_2)$ and $(\beta_1 + \beta_2).$
\end{corollary}
 
\begin{proof}
    By hypothesis, for all $x \in H$:
    \[
        \lambda_1 \|x\|^2
        \leq \sum_{n \in \mathbb{N}}
        \langle x, x_n \rangle \langle y_n, x \rangle
        \leq \beta_1 \|x\|^2,
        \qquad
        \lambda_2 \|x\|^2
        \leq \sum_{n \in \mathbb{N}}
        \langle x, u_n \rangle \langle v_n, x \rangle
        \leq \beta_2 \|x\|^2.
    \]
    Adding both inequalities:
    \begin{align*}
        (\lambda_1 + \lambda_2)\|x\|^2
        &\leq
        \sum_{n \in \mathbb{N}}
        \langle x, x_n \rangle \langle y_n, x \rangle
        +
        \sum_{n \in \mathbb{N}}
        \langle x, u_n \rangle \langle v_n, x \rangle \\
        &=
        \sum_{n \in \mathbb{N}}
        \langle x, \widetilde{x}_n \rangle
        \langle \widetilde{y}_n, x \rangle
        \leq (\beta_1 + \beta_2)\|x\|^2.
    \end{align*}
    Then $(\widetilde{\mathbb{X}}, \widetilde{\mathbb{Y}})$ is a Bessel biframe with $(\lambda_1 + \lambda_2)$ y $(\beta_1 + \beta_2).$
\end{proof}

\section{Conclusion}
\label{Sec:4}
This paper introduced bi-atomic systems for bounded operators in Hilbert spaces, which allow an operator to be represented by two Bessel sequences\eqref{sistema biatomico}, The notion is symmetric with respect to the pair $( \mathbb{X}, \mathbb{Y})$ \eqref{simetria de biatomicos}, This new concept was illustrated \eqref{ejemplo bi-atomico},   It was also shown that bi-atomic systems are dual-stable; that is, given a bi-atomic system for an operator, there exists a bi-atomic system for its adjoint.\eqref{thm:dualidad_biatomica},  On the other hand, an alternative way of showing that a pair $( \mathbb{X}, \mathbb{Y})$ is a bi-atomic system is established, using a characterization in terms of the boundedness of the coefficients and the convergence of series. \eqref{equivalencia para biatomicos}. Orthogonal projections generate new bi-atomic system \eqref{thm:proyector}, and the unitary operators that commute with the operator associated with the bi-atomic system \eqref{thm:unitario}. Furthermore, Bessel biframes yield bi-atomic systems for the identity operator \eqref{bimarco bessel implica biatomico} and the concatenation of bi-atomic systems for a pair of operators $T$ and $W$ yields a bi-atomic system for $(T+W)$\eqref{thm:concatenacion_biatomica},while preserving Bessel's bimarco property\eqref{cor:concat_bimarco}.

As future work, we ask whether it is possible to formulate biatomic systems in Krein spaces, following \cite{FerrerArroyo2023}, and whether the converse of Proposition \eqref{equivalencia para biatomicos} holds; similarly, we also to explore future extensions to spaces with $W$-metrics \cite{FAE, OGK, WFE}, as well as to relate bi-atomic systems to dual frameworks \cite{FE} and possible applications such as those developed for quantum detection, as discussed in \cite{Han}. In particular, it is of interest to determine whether the equivalent characterization of bi-atomic systems established in terms of the Douglas Theorem and the Riesz Representation Theorem can be formulated in these contexts, where the inner product ceases to be a scalar and instead becomes a function membership function or a parameterized family of sets. A possible application of this line of research is the modeling of vocal risk, following the ideas developed in \cite{vocal-risk-diagnosis}, where the two generating sequences of a bi-atomic system could represent two sources of clinical and acoustic information, respectively.
 
\section*{{Conflict of interest}}

The authors declare no conflict of interest

\end{document}